\documentclass[11pt]{amsart}
\usepackage{amsmath,amssymb,amsthm}
\usepackage[letterpaper, margin=1in]{geometry}
\usepackage{xy, enumitem}
\usepackage{graphicx}

\newtheorem{theorem}{Theorem}[section]
\newtheorem{definition}[theorem]{Definition}
\newtheorem{remark}[theorem]{Remark}

\newtheorem{proposition}[theorem]{Proposition}
\newtheorem{lemma}[theorem]{Lemma}

\newcommand{\Z}{\mathbb{Z}}

\newcommand{\C}{\mathbb{C}}

\renewcommand{\phi}{\varphi}
\renewcommand{\SS}{\mathcal{S}}
\newcommand{\TT}{\mathcal{T}}
\newcommand{\A}{\mathcal{A}}
\renewcommand{\NG}{N(\Gamma)}

\DeclareMathOperator{\cl}{cl}

\usepackage{tikz}
\usepackage[noadjust]{cite}

\begin{document}
\title{Symmetries of Spatial Graphs in $3$-manifolds}
\author{Erica Flapan \and Song Yu}
\address{Department of Mathematics, Pomona College, Claremont, CA 91711}
\address{Department of Mathematics, Columbia University, New York, NY 10027}


\thanks{The first author was supported in part by NSF Grant DMS-1607744. The second author was supported by Pomona College's Summer Undergraduate Research Program}

\date{}
\maketitle

\begin{abstract}  We consider when automorphisms of a graph can be induced by homeomorphisms of embeddings of the graph in a $3$-manifold.  In particular, we prove that every automorphism of a graph is induced by a homeomorphism of some embedding of the graph in a connected sum of one or more copies of $S^2\times S^1$, yet there exist automorphisms which are not induced by a homeomorphism of any embedding of the graph in any orientable, closed, connected, irreducible $3$-manifold. We also prove that for any $3$-connected graph $G$, if an automorphism $\sigma$ is induced by a homeomorphism of an embedding of $G$ in an irreducible $3$-manifold $M$, then $G$ can be embedded in an orientable, closed, connected $3$-manifold $M'$ such that $\sigma$ is induced by a finite order homeomorphism of $M'$, though this is not true for graphs which are not $3$-connected. Finally, we show that many symmetry properties of graphs in $S^3$ hold for graphs in homology spheres, yet we give an example of an automorphism of a graph $G$ that is induced by a homeomorphism of some embedding of $G$ in the Poincar\'e homology sphere, but is not induced by a homeomorphism of any embedding of $G$ in $S^3$.
\end{abstract}

\section{Introduction} \label{Intro} \par
The study of symmetries of spatial graphs in $S^3$ is motivated by a long history of results about symmetries of knots and links, as well as by the more recent study of symmetries of non-rigid molecules. What makes symmetries of spatial graphs distinct from those of knots and links is that for a spatial graph $\Gamma$ in $S^3$, homeomorphisms of $(S^3,\Gamma)$ can be classified by the automorphisms they induce on the underlying abstract graph.  

In fact, we can consider the automorphisms induced by homeomorphisms of a spatial graph embedded in any $3$-manifold $M$.  However, rather than focusing on which automorphisms can be induced by homeomorphisms of a given embedding $\Gamma$ in $M$, we are now interested in which automorphisms can be induced by a homeomorphism of some embedding of the graph in $M$.  That is, for each automorphism we can choose a different embedding.

\begin{definition} \label{RA} \rm{
An automorphism $\sigma$ of a graph $G$ is \textit{realizable} in a $3$-manifold $M$ if there is an embedding $\Gamma$ of $G$ in $M$ and a homeomorphism $h$ of $(M, \Gamma)$ whose restriction to $\Gamma$ is $\sigma$. In this case, we say that $h$ \textit{realizes} $\sigma$ in $M$.
}\end{definition}

 The following questions arise naturally from the above definition.

\begin{enumerate}
\item  Is every automorphism of a graph realizable in some $3$-manifold?
\medskip

\item Is every automorphism of a graph $G$ which is realizable in some $3$-manifold, necessarily realizable by a finite order homeomorphism of an embedding of $G$ in some (possibly different) $3$-manifold?
\end{enumerate}

The answer to Question (1) is ``no''  if we only consider graphs in $S^3$. In particular, it was shown in \cite{F89} that the automorphism $(1234)$ of the complete graph $K_6$ cannot be realized by any embedding of $K_6$ in $S^3$.  In fact, we prove in Section~\ref{Rigid} that the automorphism $(123)$ of the complete graph $K_7$ is not realizable in any orientable, closed, connected, irreducible $3$-manifold.   On the other hand, we show in Section~\ref{3M} that the answer to Question (1) is ``yes'', if we consider embeddings in reducible $3$-manifolds.  In particular, we prove that every graph automorphism is realizable in a connected sum of copies of $S^2\times S^1$.   

Thus non-realizability of an automorphism in $S^3$ does not imply its non-realizability in every $3$-manifold. However, as we show in Section~\ref{3M}, if an automorphism $\sigma$ of a graph is realizable in $S^3$ by an orientation preserving homeomorphism, then $\sigma$ is realizable in every $3$-manifold; and if $\sigma$ is realizable in $S^3$ by an orientation reversing homeomorphism, then $\sigma$ is realizable in every $3$-manifold that has an orientation reversing homeomorphism.

The answer to Question (2) is ``no'', as shown in the following lemma.  Thus we must be careful not to assume that every realizable automorphism can be realized by a finite order homeomorphism.

\begin{lemma}\label{123} Let $G$ be the graph illustrated in Figure ~\ref{theta} and let $\sigma=(123)$.  Then $\sigma$ is realizable in $S^3$, but $\sigma$ is not realizable by a finite order homeomorphism in any $3$-manifold.
\end{lemma}

\begin{figure}[http]
\includegraphics[width=3cm]{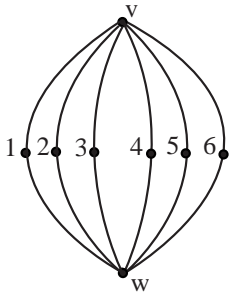}
\caption{The automorphism $(123)$ is realizable in $S^3$, but is not realizable by a finite order homeomorphism of an embedding of $G$ in any $3$-manifold.}\label{theta}
\end{figure}

\begin{proof}Let $\Gamma$ be the planar embedding of $G$ in $S^3$ illustrated in Figure~\ref{theta}.  Then $\sigma$ is realized in $S^3$ by twisting the arcs containing vertices $1$, $2$, and $3$ around vertices $v$ and $w$, while fixing the rest of the graph.

Now suppose that $G$ is embedded as $\Gamma'$ in some $3$-manifold $M$ such that $\sigma$ is realized by a finite order homeomorphism $h$ of $(M,\Gamma')$. Since $h(v)=v$ and $h$ has finite order, we can choose a neighborhood $N(v)$ such that $h(N(v))=N(v)$ and $\Gamma'$ intersects $\partial N(v)$ in six points, corresponding to the six edges incident to $v$.  Thus $h$ induces a finite order homeomorphism of the sphere $\partial N(v)$.  But since $h$ induces $(123)$ on $\Gamma'$, three of the points in $\Gamma'\cap \partial N(v)$ are fixed and three are permuted in a cycle of order three.  As this is  impossible for a finite order homeomorphism of $S^2$,  $\sigma$ cannot be realized by a finite order homeomorphism of an embedding of $G$ in any $3$-manifold $M$.\end{proof}

The Rigidity Theorem below shows that for $3$-connected graphs in $S^3$, examples like the one in Figure~\ref{theta} do not occur.  
\begin{theorem} [Rigidity Theorem for $S^3$ \cite{F95}] \label{RigidS3}
Let $G$ be a 3-connected graph. Suppose that an automorphism $\sigma$ of $G$ is realized by a homeomorphism $h$ of some embedding of $G$ in $S^3$. Then $\sigma$ is realizable in $S^3$ by a homeomorphism $f$ of finite order. Moreover, $f$ can be chosen to be orientation reversing if and only if $h$ is orientation reversing.
\end{theorem}




 In Section \ref{Rigid}, we generalize Theorem~\ref{RigidS3} to spatial graphs in homology spheres and other irreducible $3$-manifolds, which allows us to extend many previous results on symmetries of spatial graphs in $S^3$ to homology spheres. However, symmetry properties of graphs in homology spheres do not completely coincide with those in $S^3$.  In particular, in Section \ref{CompareHS} we give an example of an automorphism of a graph which is realizable in the Poincar\'e homology sphere, but is not realizable in $S^3$.

Note that by a \emph{graph}, we mean a finite, connected graph without self-loops or multiple edges between a pair of vertices, and we say a graph is {\it $3$-connected} if three or more vertices together with their incident edges must be removed in order to disconnect the graph or reduce it to a single vertex.   We work throughout in the smooth category, and require our homeomorphisms to be diffeomorphisms except possibly on the set of vertices of an embedded graph.

\medskip

\section{Realizable automorphisms of spatial graphs in $3$-manifolds}\par \label{3M}

As mentioned in Section~\ref{Intro}, not every graph automorphism is realizable in $S^3$.  But the following proposition shows that every  automorphism is realizable in some $3$-manifold.

\begin{proposition} \label{EveryAut}
Every automorphism $\sigma$ of a graph $G$ is realizable by an orientation preserving homeomorphism in a connected sum of one or more copies of $S^2\times S^1$, though such a homeomorphism will not necessarily be of finite order.
\end{proposition}

We will use the following construction of a neighborhood of a graph in a $3$-manifold throughout the paper.  Let $\Gamma$ be an embedding of a graph $G$ in a $3$-manifold $M$, with embedded sets $V$ and $E$ of vertices and edges respectively. For each vertex $v \in V$, take a ball $N(v)$ around $v$ whose intersection with $\Gamma$ is a connected set containing no vertex other than $v$. We then let $N(V) = \bigcup_{v \in V} N(v)$. For each edge $e \in E$, take a solid tube $N(e) \cong D \times I$ that has $\cl(e - N(V))$ as its core, such that $N(e) \cap N(V)$ consists of two disks $D \times \{0\}$ and $D \times \{1\}$ in $\partial N(V)$. We let $N(E) = \bigcup_{e \in E} N(e)$. Note that we can assume that the $N(v)$'s are pairwise disjoint and the $N(e)$'s are pairwise disjoint. Finally, let $\NG = N(V) \cup N(E)$, and observe that $\NG$ is a handlebody  in $M$ that contains $\Gamma$ as its spine and as such it is unique up to isotopy.
\medskip

\begin{proof}[Proof of Proposition \ref{EveryAut}]
Let $\Gamma$ denote an embedding of $G$ in $S^3$ and let $N = N(\Gamma)$ be the handlebody described above. There is a natural orientation preserving homeomorphism $g$ of $(N, \Gamma)$ that induces the automorphism $\sigma$ on $\Gamma$. Now let $M$ be the double of $N$.  Then $M$ is a closed 3-manifold. We can then view $\Gamma$ as an embedding of $G$ in $M$, and we can extend $g$ to $M-N$ by copying its behavior on $N$ to obtain an orientation preserving homeomorphism of $(M, \Gamma)$ that realizes $\sigma$.  By our construction, $M$ is a connected sum of $n$ copies of $S^2\times S^1$, where $n$ is the genus of $N$.  

Note that by the example given in Lemma~\ref{123} we know that the homeomorphism constructed above will not necessarily be of finite order. \end{proof}

Thus the non-realizability of an automorphism in $S^3$ does not imply its non-realizability in other $3$-manifolds.   On the other hand, the following proposition shows that the realizability of an automorphism in $S^3$ does imply its realizability in other $3$-manifolds, though it says nothing about the homeomorphism of the $3$-manifold having finite order.

\begin{proposition} \label{Aut3M}
Let $\sigma$ be an automorphism of a graph $G$.
\begin{enumerate}[label=(\arabic*), topsep=-5pt]
\item If $\sigma$ is realizable in $S^3$ by an orientation preserving homeomorphism, then $\sigma$ is realizable in every connected $3$-manifold by an orientation preserving homeomorphism.

\medskip

\item If $\sigma$ is realizable in $S^3$ by an orientation reversing homeomorphism, then $\sigma$ is realizable by an orientation reversing homeomorphism in every connected $3$-manifold that possesses an orientation reversing homeomorphism.
\end{enumerate}
\end{proposition}

\begin{proof}Let $\Gamma$ be an embedding of $G$ in $S^3$ and $h$ be a homeomorphism of $(S^3, \Gamma)$ that realizes $\sigma$. Then $h$ is isotopic to a homeomorphism of $(S^3, \Gamma)$ that induces $\sigma$ on $\Gamma$ and fixes some point $p$ in $S^3 - \Gamma$. Thus, we can assume without loss of generality that $h$ setwise fixes a closed ball around $p$ that is disjoint from $\Gamma$. Let $B$ be the closure of the complement of this ball in $S^3$ and let $S = \partial B$. Then $\Gamma \subseteq B$, $h(B) = B$, and $h(S) = S$; and since $h$ does not interchange the two components of $S^3-S$, the restriction $h|_S$ is orientation reversing if and only if $h$ is orientation reversing.

Let $M$ be a connected orientable $3$-manifold and let $f$ be a homeomorphism of $M$ that is orientation reversing if and only if $h$ is orientation reversing. As we saw above for $h$, without loss of generality, we can assume that $f$ setwise fixes a closed ball $B'$ in $M$ and restricts to a homeomorphism of $S' = \partial B'$ which is orientation reversing on $S'$ if and only if $f$ is orientation reversing on $M$.  Thus $h|_S$ is orientation reversing if and only if $f|_{S'}$ is orientation reversing.

Let $S'\times I$ denote a collar of $S'$ in $B'$ with $S'\times \{1\}=S'$, and let $B''$ be the closed ball in $B'$ whose boundary is $S'\times \{0\}$.  Then there is a homeomorphism $\psi:B''\to B$.  Now $\psi^{-1}\circ h\circ\psi:B''\to B''$ is a homeomorphism  which is orientation reversing if and only if $h$ is orientation reversing.  In particular, the two homeomorphisms $\psi^{-1}\circ h\circ\psi|_{\partial B''}$ and $f|_{S'}$ are self-maps of $2$-spheres which are either both orientation preserving or both orientation reversing. Thus they are isotopic.

We now define a homeomorphism $g: M \to M$ that agrees with $\psi^{-1}\circ h\circ\psi$ on $B''$ and agrees with $f$ on $\mathrm{cl}(M-B')$, and by using the above isotopy we extend $g$ within the collar $S' \times I$. Thus $\Gamma'=\psi^{-1}(\Gamma)$ is an embedding of $G$ in $B'' \subseteq M$ that is setwise fixed by $g$, and $g|_{\Gamma'}$ induces the automorphism $\sigma$ on $\Gamma'$.  Furthermore, $g$ is orientation preserving if and only if $h$ is orientation preserving.
\end{proof}
\medskip


\medskip

\section{Rigidity of symmetries of spatial graphs} \label{Rigid}\par
In this section, we generalize the Rigidity Theorem to homology spheres and irreducible $3$-manifolds.  The goal is to show that under suitable conditions, realizable automorphisms can be realized by finite order homeomorphisms though the embedding and even the manifold in which the graph is embedded may need change.  We begin by defining three types of spheres relative to a spatial graph in a $3$-manifold (see Figure \ref{SepSph})

\begin{definition} \rm{
Let $\Gamma$ be a spatial graph in a $3$-manifold $M$.
\begin{enumerate}[label=(\arabic*), topsep=-5pt]
\item A \textit{type I sphere} of $\Gamma$ is a sphere $S$ embedded in $M$ that satisfies:
	\begin{enumerate}[label=(\roman*), topsep=-2pt]
		\item $S\cap \Gamma$ is a single vertex;
		\item each of the components of $M -S$ has non-empty intersection with $\Gamma$.
	\end{enumerate}
	\medskip

\item A \textit{type II sphere} of $\Gamma$ is a sphere $S$ embedded in $M$ that satisfies:
	\begin{enumerate}[label=(\roman*), topsep=-2pt]
		\item $S\cap\Gamma$ is two vertices;
		\item the closure of each component of $M - S$ contains at least two edges of $\Gamma$;
		\item the annulus $S \cap \cl(M-\NG)$ is incompressible in $\cl(M-\NG)$.
	\end{enumerate}
	\medskip

\item A \textit{type III sphere} is a pinched sphere $S$ embedded in $M$ that satisfies:
	\begin{enumerate}[label=(\roman*), topsep=-2pt]
		\item $S\cap\Gamma$ is a single vertex which is also the pinch point of $S$;
		\item each of the components of $M -S$ has non-empty intersection with $\Gamma$;			\item the annulus $S \cap \cl(M-\NG)$ is incompressible in $\cl(M-\NG)$.
	\end{enumerate}
\end{enumerate}
}\end{definition}

Note, we abuse notation in (3) by referring to $S$ as a type III {\it sphere} rather than a {\it pinched sphere}.

\begin{figure}[h]
	\begin{center}
		\includegraphics[scale=0.5]{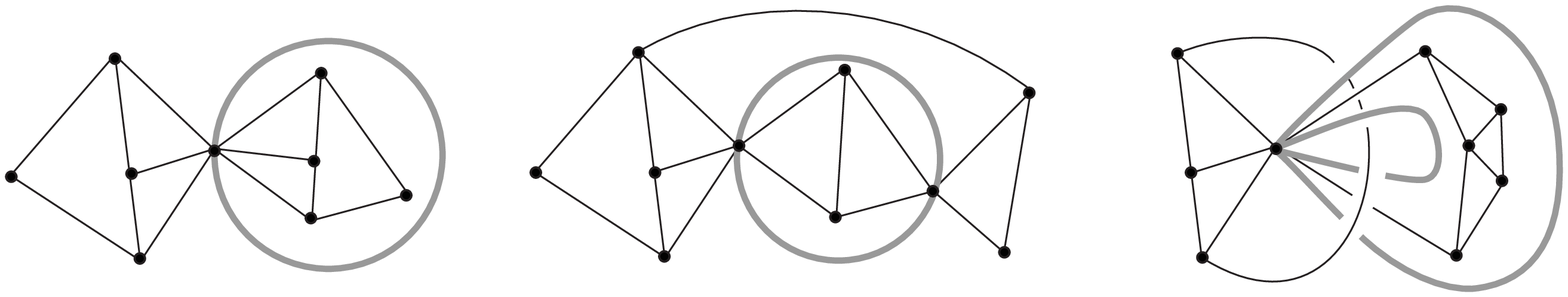}
	\end{center}
	\caption{Spatial graphs with type I, II, and III spheres in grey.}
	\label{SepSph}
\end{figure}

\medskip

\begin{remark} \label{SphVal} \rm{
Suppose that a graph $G$ has at least two edges.  If $G$ has a vertex of valence one, then every embedding of $G$ in a $3$-manifold has a type I sphere; and if $G$ has a vertex of valence two, then every embedding of $G$ in a $3$-manifold has a type II sphere (See Figure \ref{LowVal}). In particular, if the handlebody $\NG$ has genus less than 2, then every embedding of $G$ in any $3$-manifold has a sphere of type I or II.
}\end{remark}

\begin{figure}[h]
\[	\includegraphics[scale=0.5]{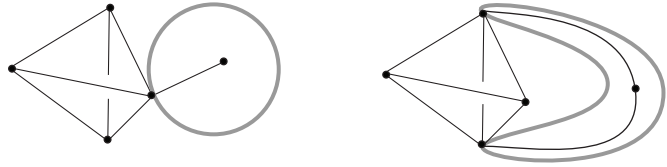}
\]
	\caption{Vertices of valence $1$ or $2$ can be separated by spheres of type I or II}
	\label{LowVal}
\end{figure}


\begin{theorem}[Rigidity Theorem] \label{RigidHS}
Let $M$ be an orientable, closed, connected, irreducible $3$-manifold, let $\Gamma$ be an embedding of a graph $G$ in $M$ that does not have any spheres of type I, II, or III, and let $\sigma$ be an automorphism of $G$ that is realized by a homeomorphism $h$ of $(M, \Gamma)$.  Then the following hold.
\begin{enumerate}[label=(\arabic*), topsep=-5pt]
\item If $M=S^3$, then $\sigma$ is realizable in $S^3$ by a homeomorphism $f$ of finite order;\par
\item If $M$ is a homology sphere, then $\sigma$ is realizable in some homology sphere by a homeomorphism $f$ of finite order;\par
\item Otherwise, $\sigma$ is realizable in some orientable, closed, connected $3$-manifold by a homeomorphism $f$ of finite order;\par
\end{enumerate}
\medskip

\noindent Moreover, $f$ can be chosen to be orientation reversing if and only if $h$ is orientation reversing.
\end{theorem}

Note that in contrast with Proposition 2.2, Theorem \ref{RigidHS} guarantees that $\sigma$ is induced by a finite order homeomorphism. 

\begin{remark} \label{Sph3C} \rm{Suppose that $\Gamma$ has a type I, II, or III sphere $S$ in a $3$-manifold $M$. Since each component of $M -S$ has non-empty intersection with $\Gamma$, removing the (one or two) vertices in $S \cap \Gamma$ along with any incident edges disconnects $\Gamma$.  Thus $\Gamma$ cannot be 3-connected.  Hence Theorem~\ref{RigidHS} holds for any $3$-connected graph, regardless of the embedding of the graph in $M$.}\end{remark}

Observe that the embedding on the left in Figure \ref{N3C} provides a counterexample to the converse of the above remark, where the graph is not even 2-connected but the embedding has no type I, II, or III sphere.  {Furthermore, the automorphism $(15)(26)(37)$ is induced on the embedding on the left by turning the figure over and then sliding the knot back to the left side.  However, this homeomorphism does not have finite order.  By contrast $(15)(26)(37)$ is induced on the embedding on the right by the involution which turns the graph over.  Note that the embedding on the right has a type I sphere.  This illustrates why in the conclusion of Theorem~\ref{RigidHS}, we cannot require that the new embedding have no spheres of type I, II, or III.

\begin{figure}[h]
\[	\includegraphics[scale=0.4]{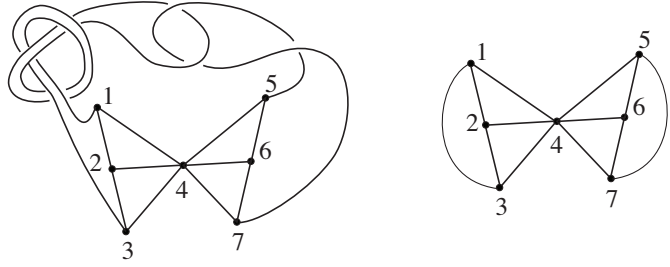}
\]
	\caption{ $(15)(26)(37)$ is realized by a homeomorphism (not of finite order) of the embedding on the left.  While $(15)(26)(37)$ is realized by a finite order homeomorphism of the embedding on the right.}
	\label{N3C}
\end{figure}

For the proof we will need two lemmas, the first of which makes use of the well known Half Lives, Half Dies Theorem with integer coefficients (see Hatcher's notes \cite{H}).

\begin{lemma}\label{TorusHS}
Let $M$ be a homology sphere containing a torus $T$, let $N_1$ and $N_2$ denote the closed up components of $M-T$, and let $i_*:H_1(T)\to H_1(N_1)$ and $j_*:H_1(T)\to H_1(N_2)$ be the inclusion maps.  Then there are simple closed curves $\lambda_1$ and $\lambda_2$ in $T$, which are unique up to isotopy, such that each $\lambda_k$ bounds a surface in $N_k$.  Furthermore,  $j_*(\lambda_1)$ and $i_*(\lambda_2)$ are generators of $H_1(N_2)=\Z$ and $H_1(N_1)=\Z$ respectively.
\end{lemma}

We omit the proof of this lemma since it is quite standard.

\medskip

\begin{remark} \label{SewTorus} \rm{
As a consequence of Lemma \ref{TorusHS}, if we sew a solid torus to $N_1$ along $T$ by gluing its meridian along the generator $\lambda_2$, the resulting closed $3$-manifold is a homology sphere.
}\end{remark}

\medskip

\begin{lemma} \label{IrrLem}
Let $\Gamma$ be an embedding of a graph $G$ in a homology sphere $M$ which does not have any spheres of type I, II, or III, and suppose that $\sigma$ is an automorphism of $G$ that is realized by a homeomorphism $h$ of $(M,\Gamma)$. Then there exists an embedding $\Gamma'$ of $G$ in a homology sphere $M'$ and a homeomorphism $f$ of $(M',\Gamma')$ such that
\begin{enumerate}[label=(\roman*), topsep=-5pt,itemsep=0pt]
\item $\cl(M' - N(\Gamma'))$ is irreducible;\par
\item $\Gamma'$ does not have any spheres of type I, II, or III in $M'$;\par
\item $f$ realizes $\sigma$;\par
\item $f$ is orientation reversing if and only if $h$ is orientation reversing.
\end{enumerate}
\end{lemma}

\begin{proof}
Let $\NG$ be the neighborhood of $\Gamma$ in $M$ defined in Section \ref{3M}. Since $\Gamma$ is setwise invariant under $h$, we can modify $h$ by isotopy (and still refer to the resulting map as $h$) such that $h(\NG) = \NG$, $h(N(v)) = N(\sigma(v))$ for each vertex $v$, and $h(N(e)) = N(\sigma(e))$ for each edge $e$. As a result, $h$ restricts to a homeomorphism of $(\cl(M-\NG), \partial \NG)$. 

We now apply Theorem 3.1 of \cite{B02} to $\cl(M-\NG)$ to get a finite family $\SS$ of disjoint separating spheres, which decompose $\cl(M-\NG)$ into a connected sum of prime $3$-manifolds such that the set of summands is unique up to diffeomorphism.  Furthermore, the closure of one component of $M-\NG -\SS$ is a $3$-sphere with a finite number of open balls removed, and the closures of the other components each contain only one element of $\SS$ in their boundary.  Let $M_1$ denote the closure of the component of $M-\NG -\SS$ that contains $\partial \NG$, and let $M_2$ denote the irreducible $3$-manifold formed from $M_1$ by filling its single spherical boundary component with a ball $B$.

Since $h$ is a homeomorphism of $(\cl(M-\NG), \partial \NG)$, the family of spheres $h(\SS)$ also provides a prime decomposition of $\cl(M-\NG)$ with the same properties as $\SS$. Now since $\cl(M-\NG)$ is orientable, Theorem 3.1 of \cite{B02} implies that $\cl(M-\NG)$ has the stronger uniqueness property (proved in Appendix A of \cite{B83}) that tells us there is an orientation preserving homeomorphism $h'$ of $\cl(M-\NG)$ such that $h'(h(\SS)) = \SS$ and $h'$ pointwise fixes $\partial \NG$.  Hence $h'\circ h$ coincides with $h$ on $\partial \NG$.  Also, since $M_1$ contains $\partial \NG$, we must have $h'(h(M_1))=M_1$.  We can now extend $(h' \circ h)|_{M_1}$ radially into the ball attached at the spherical boundary component of $M_1$ to obtain a homeomorphism $h_2$ of $M_2$ such that $h_2$ is orientation reversing if and only if $h$ is orientation reversing and $h_2$ coincides with $h$ on $\partial \NG$.

We obtain a closed $3$-manifold $M'$, by attaching $\NG$ to $M_2$ along $\partial \NG = \partial M_2$ as it was originally attached to $M_1$.  We show below that $M'$ is a homology sphere. Note that $M_1 \cup \NG$ is a closed up component of $M - (\partial M_1 - \partial \NG)$. Now consider the Mayer-Vietoris sequence applied to $M$ and subspaces $M_1 \cup \NG$, $\cl(M - (M_1 \cup \NG))$:
\[	\cdots  \to H_1(\partial M_1 - \partial \NG) \to H_1(M_1 \cup \NG) \oplus H_1(\cl(M- (M_1 \cup \NG))) \to H_1(M) \to \cdots.	\]
Since $H_1(M) = 0$ and $\partial M_1 - \partial N(\Gamma)$ consists of spheres, we have $H_1(M_1 \cup \NG) = 0$. As $M'$ can be obtained from $M_1 \cup \NG$ by filling its boundary components with balls, it follows that $H_1(M') = 0$. By Poincar\'e duality, $H_2(M') = 0$, and thus $M'$ is a homology sphere.  Furthermore, $\cl(M' - N(\Gamma'))=M_2$ is irreducible.

Our embedding $\Gamma$ of $G$ in $M$ induces an embedding $\Gamma'$ of $G$ in $M'$. Suppose for the sake of contradiction that $\Gamma'$ has a sphere $S$ of type I, II, or III in $M'$. We can isotop $S$ off of the ball $B$ to obtain a sphere $S'$ of type I, II, or III for $\Gamma'$ that is contained in $M_1 \cup \NG$, and hence in $M$.  Since $M$ is a homology sphere $H_2(M)=0$, and hence $M$ cannot contain a non-separating sphere.  Thus $S'$ is a separating sphere in $M$.  It follows that $S'$ is a type I, II, or III sphere for $\Gamma$ in $M$, contrary to hypothesis.

Finally, since $h_2$ coincides with $h$ on $\partial \NG$, we can extend $h_2$ to $M'$ according to how $h$ maps $\NG$ to itself to obtain a homeomorphism $f$ of $(M', \Gamma')$ that realizes $\sigma$. Note that $f$ is orientation reversing if and only if $h$ is orientation reversing.
\end{proof}\medskip

For convenience, we introduce some notation based on the construction of the neighborhood $N(\Gamma)$ given in Section~2.  For each $v \in V$, let $\partial'N(v)$ denote the sphere with holes $\partial N(v) \cap \partial \NG$. For each $e \in E$, let $\partial'N(e)$ denote the annulus $\partial N(e) \cap \partial \NG$. Finally, let $\partial'N(V) = \bigcup_{v \in V} \partial'N(v)$ and $\partial' N(E) = \bigcup_{e \in E} \partial'N(e)$. Then $\partial \NG = \partial'N(V) \cup \partial'N(E)$.

\begin{proof}[Proof of Theorem~\ref{RigidHS}]
Since the proofs of the three parts are lengthy but similar, we will simultaneously prove all three parts and point out the differences among the three parts along the way.  To begin with, we let $\NG$ be the neighborhood of $\Gamma$ in $M$ defined in Section \ref{3M}. Regardless of whether we are in Part (1), (2), or (3), if $G$ consists of a single edge, then we can embed $G$ as a line segment in $S^3$ and $\sigma$ will be induced by a finite order homeomorphism of $S^3$.  Thus we assume $G$ has at least two edges.  Recall from Remark \ref{SphVal} that $G$ cannot have any vertices of valence less than 3. It follows that the handlebody $\NG$ has genus at least 2.

For part (1) of the Theorem, since $G$ is connected, $\cl(M-\NG) = \cl(S^3-\NG)$ is irreducible. For part (2), using Lemma \ref{IrrLem} to change the homology sphere $M$ if necessary, we can assume without loss of generality that $\cl(M - \NG)$ is irreducible. For part (3), first suppose that $\Gamma$ is contained in a ball $B$ in $M$. We can choose $B$ so that it is invariant under the homeomorphism $h$ of $(M,\Gamma)$.  Then we can embed $\Gamma\subseteq B$ in $S^3$, to get an embedding $\Gamma'$ of $G$ in $S^3$ that does not have any spheres of type I, II, or III. Now, applying the argument of the proof of Proposition \ref{Aut3M}, we can use $h|_B$ to construct a homeomorphism $h'$ of $(S^3, \Gamma')$ that realizes $\sigma$. In this case, it suffices to prove part (1) of the Theorem. Thus, we will assume for part (3) that $\Gamma$ is not contained in a ball in $M$. Since $M$ is irreducible, this means that $\cl(M-\NG)$ is also irreducible.

Hence for all three parts of the Theorem, we can assume that $\cl(M-\NG)$ is irreducible. Thus we can apply the JSJ Characteristic Tori Decomposition (see \cite{JS79}, \cite{J79}) to get a minimal family $\TT$ of disjoint, incompressible tori in $\cl(M - \NG)$ such that each closed up component of $\cl(M - \NG) - \TT$ is Seifert fibered or atoroidal. Moreover, the family $\TT$ is minimal and unique up to isotopy.  Thus we can modify $h$ by an isotopy (and abuse notation by again referring to the resulting map also as $h$) such that $h(\TT) = \TT$, $h(\Gamma) = \Gamma$, $h(N(v)) = N(\sigma(v))$ for each $v \in V$, and $h(N(e)) = N(\sigma(e))$ for each $e \in E$. Now let $X$ be the closed up component of $\cl(M-\NG) - \TT$ that contains $\partial \NG$. Then $h(X) = X$.

Since $\partial \NG$ has genus at least 2, $X$ cannot be Seifert fibered and hence is atoroidal. We show as follows that $X$ is irreducible. Since $\cl(M - \NG)$ is irreducible, it suffices to consider the case $X \neq \cl(M - \NG)$.  Thus we assume that $\TT$ is non-empty. Let $S$ be a sphere in $X\subseteq \cl(M - \NG)$. Then the irreducibility of $\cl(M - \NG)$ implies that $S$ splits $\cl(M - \NG)$ into two closed up components, one of which is a ball.  One such component contains $\cl((M - \NG)- X)$ and the other such component is contained in $X$.  The component containing $\cl((M - \NG)- X)$ must contain some torus in $\TT$.  Since these tori are incompressible, $\cl((M - \NG)- X)$ cannot be a ball.  This implies that the closed up component of $\cl(M - \NG) - S$ contained in $X$ must be a ball. Thus $X$ is irreducible.

Now, let $P$ denote the union of the torus boundary components of $X$ together with the annuli in $\partial'N(E)$. Then, $P \subseteq \partial X$ and $\cl(\partial X -P) = \partial' N(V)$ consists of spheres with holes. Suppose, for the sake of contradiction, that there is a compressing disk $D$ for $\partial X-P$ in $X$. Then $\partial D$ is a non-trivial loop in $\partial'N(v)$, for some vertex $v \in V$. Hence the two components of $\partial'N(v) - \partial D$ each have at least one hole through which an edge of $\Gamma$ passes. Thus, if we take a disk $D'$ in $N(v)$ such that $\partial D' = \partial D$ and $D' \cap \Gamma = \{v\}$, then the sphere $D \cup D'$ would be a type I sphere of $\Gamma$, which cannot exist by our hypothesis. It follows that $\partial X-P$ is incompressible.

Since $X$ is irreducible and $\partial X-P$ is incompressible, we can apply the JSJ Characteristic Submanifold Decomposition for Pared Manifolds (see \cite{JS79}, \cite{J79}) to the pared manifold $(X,P)$ to obtain a minimal family $\A$ of incompressible tori and annuli in $X$ with boundaries in $\partial X-P$ such that for each closed up component $W$ of $X - \A$, the pared manifold $(W, W \cap (P \cup \A))$ is either simple, Seifert fibered, or $I$-fibered. Since $X$ is atoroidal, $\A$ cannot contain any tori. By the uniqueness of $\A$ up to isotopy, this means that we can again modify $h$ by an isotopy (and, by an abuse of notation, refer to the resulting map as $h$) so that $h(\A) = \A$.

Let $A$ be one of the annuli in $\A$.  Suppose, for the sake of contradiction, that there is a compression disk for some $A$ in $\cl(M - \NG)$. Since $A$ is incompressible in $X$, this disk must meet some torus in $\TT$ non-trivially, giving us a compression disk for the torus in $\cl(M-\NG)$. This contradiction shows that $A$ is incompressible in $\cl(M-\NG)$.

Suppose that there are distinct vertices $v_1$ and $v_2$ such that one component of $\partial A$ is contained in $\partial' N(v_1)$ and the other is contained in $\partial' N(v_2)$.  Let $D_1\subseteq N(v_1)$ and $D_2\subseteq N(v_2)$ be disks intersecting $\Gamma$ at $v_1$ and $v_2$ respectively, such that $\partial A = \partial D_1 \cup \partial D_2$. Then $S=A \cup D_1 \cup D_2$ is a sphere that separates $M$ into closed up components $U$ and $U'$. If both $U$ and $U'$ contain at least two edges, $S$ would be a type II sphere, contrary to hypothesis. Thus, without loss of generality, $\Gamma\cap U$ contains at most one edge.  However, since $A$ is incompressible, in fact, $\Gamma\cap U$ consists of exactly one edge $e$.

Now since $e$ is the only edge of $\Gamma$ whose vertices are $v_1$ and $v_2$, we can define the set $\A_e$ of all annuli in $\A\cup P$ with one boundary in $\partial N(v_1)$ and the other boundary in $\partial N(v_2)$.  We can then cap off each annulus in $\A_e$ with disks in $\partial N(v_1)$ and $\partial N(v_2)$ to obtain a collection of spheres.  Since $M$ is either $S^3$, a homology sphere, or an irreducible manifold, each such sphere separates $M$ into two closed up components; and the set of such components which contain $e$ are nested.  Thus we can let $U_e$ denote the outermost component with respect to this nesting.

We do the above construction for each annulus in $\A\cup P$ with boundaries in distinct components of $\partial N(V)$ to obtain a collection of solids $U_{e_1}$, \dots, $U_{e_n}$ which each meet precisely two components of $\partial N(V)$. Observe that for every edge $e_i$, the annulus $\partial'N(e_i)$ is in $P$, and hence every edge $e_i$ in $\Gamma$ is contained in some $U_{e_i}$.  Also, every $\partial U_{e_i}$ meets $\Gamma$ only in the two points of $e_i\cap \partial N(V)$.

Next suppose that for some vertex $v$, the boundaries of the annulus $A$ are both in $\partial N(v)$ and $A$ is not contained in one of the $U_{e_i}$.  Let $D_1$ and $D_2$ be disks in $N(v)$ such that $\partial A = \partial D_1 \cup \partial D_2$ and for each $i$ we have $D_i\cap \Gamma=\{v\}$. Then $A\cup D_1 \cup D_2$ is a pinched sphere that separates $M$ into closed up regions $F_A$ and $F'_A$. If both $F_A$ and $F'_A$ contain at least one edge, then this pinched sphere would be a type III sphere.  Thus without loss of generality, $F_A\cap \Gamma=\{v\}$.  Now the boundary components of $A$ divide $\partial N(v)$ into two disks and an annulus.  If either of these disks were disjoint from $\Gamma$ then $A$ would be compressible in $X$.  Since $F_A\cap \Gamma=\{v\}$, it follows that $F_A\cap \partial N(v)$ is an annulus which is disjoint from $\Gamma$.  In other words, the boundary components of $A$ co-bound an annulus in $\partial'N(v)$.

We now cap off all of the annuli in $\A$ which have both boundaries in $\partial N(v)$ and are not contained in any $U_{e_i}$ to obtain a collection of pinched spheres which bound $3$-dimensional regions whose intersection with $\Gamma$ consists of the single vertex $v$.  Since the set of such regions containing a given annulus $A$ are nested, we can choose $A\in\A$ such that $F_A$ is an outermost such region.  We do this construction for all those annuli in $\A$ with both boundaries in the same component of $\partial N(V)$ which are not contained in any of the $U_{e_i}$.

In this way, we obtain outermost regions $F_{A_1}$, \dots, $F_{A_m}$ and $U_{e_1}$\dots $U_{e_n}$ such that all of the annuli in $\A\cup P$ are contained in $U_{e_1}\cup\dots\cup U_{e_n}\cup F_{A_1}\cup \dots\cup F_{A_m}$.  Also, the regions $F_{A_1}\cap X$, \dots, $F_{A_m}\cap X$ and $U_{e_1}\cap X$, \dots, $U_{e_n} \cap X$ are pairwise disjoint.  See Figure~\ref{annuli}.

\begin{figure}[h]
	\begin{center}
		\includegraphics{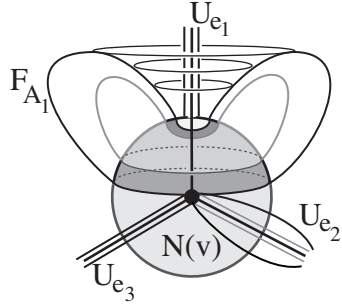}
	\end{center}
	\caption{An illustration of some of the outermost regions $U_{e_i}$ and $F_{A_j}$.}
	\label{annuli}
\end{figure}

Let $W=\mathrm{cl}(X-(U_{e_1}\cup\dots\cup U_{e_n}\cup F_{A_1}\cup \dots\cup F_{A_m}))$, as illustrated in Figure~\ref{Wnew} where $T$ is one of the torus boundary components of $X$. Observe from our construction that $W$ is the closure of a single component of $X-(\A\cup P)$.  Hence by the JSJ Decomposition  for pared manifolds, the pared manifold $(W, W \cap (P \cup \A))$ is either simple, Seifert fibered, or $I$-fibered.

Since the boundary components of the annuli in $\A\cup P$ are pairwise disjoint curves in $\partial N(V)$, no component of $X\cap \partial N(V)$ is entirely contained in $U_{e_1}\cup\dots\cup U_{e_n}\cup F_{A_1}\cup \dots\cup F_{A_m}$.  Thus, $W$ meets every component of $\partial N(V)$.  Also, for every pair of vertices $v_1$ and $v_2$ which are joined by an edge, there is an annulus in $\partial W$ with boundaries in $\partial N(v_1)\cap W$ and $\partial N(v_2)\cap W$.  Since $\partial \NG$ has genus at least 2, this means that the boundary component of $W$ which meets $\partial \NG$ has genus at least 2.  It follows that $(W, W \cap (P \cup \A))$ is not Seifert fibered.

\begin{figure}[h]
	\begin{center}
		\includegraphics{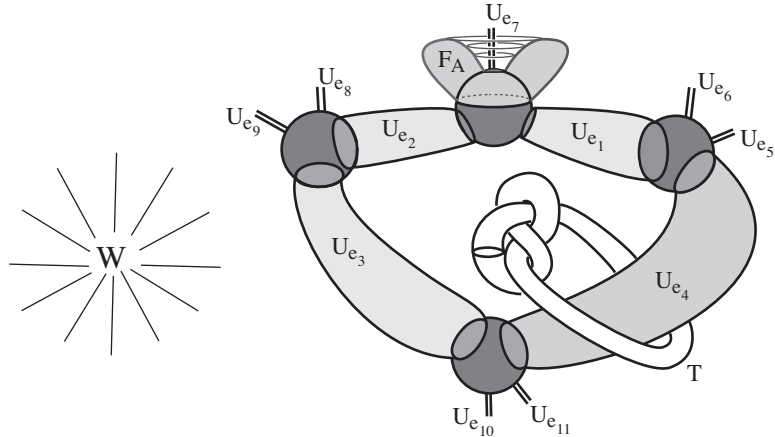}
	\end{center}
	\caption{$W=\mathrm{cl}(X-(U_{e_1}\cup\dots\cup U_{e_n}\cup F_{A_1}\cup \dots\cup F_{A_m}))$.}
	\label{Wnew}
\end{figure}

Next suppose that $(W, W \cap (P \cup \A))$ is $I$-fibered as a pared manifold. Then there is an $I$-bundle map of $W$ over a base surface $S$ such that $W \cap (P \cup \A)$ is in the pre-image of $\partial S$. This implies that $S$ must be homeomorphic to a component of $\partial' N(V) \cap W$. Hence $S$ is a sphere with holes. Since $W$ is orientable and is $I$-fibered over the orientable surface $S$, it follows that $W$ is homeomorphic to $S \times I$. Hence $W \cap (P \cup \A) = \partial S \times I$, and $S \times \{0\}$ and $S \times \{1\}$ are the only components of $\partial'N(V) \cap W$. But we saw above that $W$ meets every component of $\partial N(V)$.  Thus $G$ has at most two vertices, contradicting our earlier observation that every vertex of $G$ has valence at least $3$. Hence $(W, W \cap (P \cup \A))$ cannot be $I$-fibered.

It now follows that the pared manifold $(W, W \cap (P \cup \A))$ is simple. We can thus apply Thurston's Hyperbolization Theorem \cite{T82} to $(W, W \cap (P \cup \A))$ to equip $W - (P \cup \A)$ with a finite volume complete hyperbolic metric with totally geodesic boundary.

Recall that $h(X) = X$ and $h(\A) = \A$. Observe from our construction of $F_{A_1}$, \dots, $F_{A_m}$ and $U_{e_1}$\dots $U_{e_n}$ that $W$ is the only closed up component of $X - \A$ that meets more than two components of $\partial N(V)$. Thus we must have $h(W) = W$, and hence $h$ restricts to a homeomorphism of $(W, W \cap (P \cup \A))$. 

Let $D$ denote the double of $W - (P \cup \A)$ along its boundary.  Then $D$ is a finite volume hyperbolic manifold, and we can extend $h$ to a homeomorphism $h'$ of $D$. Using the work of Thurston \cite{Thu}, we can now apply Mostow's Rigidity Theorem \cite{M73} to $D$, to obtain a finite order isometry $f'$ of $D$ such that $f'$ is homotopic to $h'$ in $D$.  

Let $\varphi$ be the isometric involution that interchanges the two halves of $D$. By construction, $h'$ commutes with $\varphi$. The isometry $f'$ will also commute with $\varphi$, because the commutator $f' \varphi (f')^{-1} \varphi^{-1}$ is an isometry that is homotopic to the identity, and therefore must be the identity. As a consequence, $f'$ sends the fixed point set of $\varphi$ to itself, namely $f'$ sends the boundary of $W - (P \cup \A)$ to itself. Since $f'$ cannot interchange the two halves of $D$ for homotopy reasons, it must send $W$ to itself. We can now restrict $f'$ to $W$, to obtain a homeomorphism $f$ of $(W, W \cap (P \cup \A))$ that has finite order and is homotopic to $h$ by a homotopy $H$ of $D$. 

 In order to prove that $f$ is homotopic to $h$ by a homotopy of $W$, we consider the fundamental group of $W$.  In particular, let $x_0$ be a base point in $W$, and let $\ell\subseteq D$ be the path from $h(x_0)$ to $f(x_0)$ which is induced by $H$.  Now $h_*:\pi_1(W,x_0)\to\pi_1(W, h(x_0))$ and $f_*:\pi_1(W,x_0)\to\pi_1(W, f(x_0))$ are conjugate by $\ell$ in $D$. However, because $(W, W \cap (P \cup \A))$ is anannular, the free homotopy between $h(\gamma)$ and $f(\gamma)$ in $D$ can be homotoped inside $W$ by Johansson \cite{J79}.  Therefore the homotopy $H$ can be chosen such that $\ell$ is contained in $W$, and hence $h_*$ and $f_*$ are actually conjugate by $\ell$ in $W$.  Since $\pi_2(W)$ and $\pi_3(W)$ are trivial, it follows that $f$ and $h$ are homotopic in $W$. 

Now, $f$ induces isometries on the collection of tori and annuli in $W\cap (P\cup \A )$ with respect to a flat metric \cite{Thu}.  Finally, Waldhausen's Isotopy Theorem \cite{W68} ensures that $f$ is actually isotopic to $h$ by an isotopy of $W$ that leaves $W \cap (P \cup \A)$ setwise invariant.

Below, we fill in the boundary components of $W$ and extend $f$ to a finite order homeomorphism of the closed manifold we obtain.  We will abuse notation throughout by referring to the extended maps as $f$.  We begin by extending $f$ to $N(V)$ as follows.  Observe that the components of $\partial N(V)-W$ consist of disks in the $U_{e_i}$ which each meet $\Gamma$ in a single point and annuli in the $F_{A_j}$ which are disjoint from $\Gamma$.  We extend $f$ radially within these disks.  For the annuli in $\partial N(V)-W$, recall that $f$ is a finite order isometry of the annuli $A_1$, \dots, $A_m$ taking the set $\{\partial A_1,\dots, \partial A_m\}$ to itself.  Thus we can extend $f$ to a finite order isometry of the annuli in $\partial N(V)-W$.  In this way,  we have extended $f$ to a finite order homeomorphism of $\partial N(V)$. Thus we can now extend $f$ radially to $N(V)$, re-embedding $N(V)\cap \Gamma$ if necessary, so that $f$ takes $N(V)\cap \Gamma$ to itself and $f$ is still of finite order.  Furthermore, since $f$ is isotopic to $h$ on $\partial W$, it follows from our constructions of these extensions that $f$ is isotopic to $h$ on $W\cup N(V)$.

Next, for each $i$, we let $S_{e_i}$ denote the sphere $\partial(\mathrm{cl}(U_{e_i}-N(V)))$.  For each $i$, let $D_i$ be a disk and attach the solid tube $D_i\times I$ to $S_{e_i}$ such that $D_i\times \{0\}$ and $D_i\times \{1\}$ are the components of $S_{e_i}\cap \partial N(V)$.  Now, let $C_i$ denote a core of the solid tube $D_i\times I$ whose endpoints coincide with the points of $ \partial N(V)\cap \Gamma$.  Then $\Gamma'=(N(V)\cap \Gamma)\cup C_1\cup\dots\cup C_n$ is an embedding of $\Gamma$ in $Y=W\cup N(V)\cup (D_1\times I) \cup\dots\cup (D_n\times I)$ and $N(\Gamma')= N(V)\cup (D_1\times I) \cup\dots\cup (D_n\times I)$ is a neighborhood of $\Gamma'$ in $Y$.  Thus $Y=W\cup N(\Gamma')$.

Since $f$ takes the sets $\{U_{e_1}, \dots, U_{e_n}\}$ and $N(V)$ to themselves, we can extend $f$ as a product map into the $D_i\times I$ so that $f$ takes the set $\{C_1,\dots, C_n\}$ to itself and $f$ is still of finite order.  Also, since $W$ meets every component of $N(V)$, and $f$ is isotopic to $h$ on $W$, $f$ induces the same permutation of the components of $N(V)$ as $h$.  Thus $f$ induces the automorphism $\sigma$ on $\Gamma'$.

From here, we make separate arguments for the different parts of the theorem.  For part (3), we obtain the manifold $M'$ by taking the double of $Y$ along its boundary, and extending $f$ in the natural way.  Then $f$ has finite order, induces $\sigma$ on $\Gamma'$, and is orientation reversing if and only if $h$ is.

Next we assume that $M$ is $S^3$ or a homology sphere. Recall that for each $j$, $\partial F_{A_j}$ is a pinched sphere meeting $\partial' N(V)$ in an annulus.  Thus $\partial (\mathrm{cl}(F_{A_j}-N(V)))$ is a torus consisting of the union of the annulus $A_j$ and the annulus in $\partial' N(V)$.  Let $R_1, R_2, \dots, R_p$  denote the boundary components of $Y$.  Then all of the $R_i$ are tori, with some contained in $\TT$ and others of the form $\partial (\mathrm{cl}(F_{A_j}-N(V)))$.

Now we apply Lemma \ref{TorusHS} to each torus $R_j$ in $M$, to obtain a simple closed curve $\lambda_j$  which is unique up to isotopy in $R_j$, and such that $\lambda_j$ bounds a surface in the closed up component of $M - R_i$ that is disjoint from $\NG$. Observe that since $h$ leaves $\TT$ and $\A$ setwise invariant, $h$ also leaves the set $\{R_1,\dots,R_p\}$ invariant.  Thus $h$ takes the set of curves $\{\lambda_1,\dots, \lambda_p\}$ to an isotopic set of curves on the tori $R_1$, \dots, $R_p$.  Now since $f$ is isotopic to $h$ on $W\cup N(V)$, it follows that $f$ also takes the set of curves $\{\lambda_1,\dots, \lambda_p\}$ to an isotopic set of curves on the tori $R_1$, \dots, $R_p$.  However, since $f$ is a finite order isometry of $R_1$, \dots, $R_p$, there are simple closed curves $\ell_1$, \dots, $\ell_p$ on the tori $R_1$, \dots, $R_p$ respectively, such that each $\ell_j$ is isotopic to $\lambda_j$ and $f$ takes the set $\{\ell_1, \dots, \ell_p\}$ to itself.  For each $j$, we sew in a solid torus $Z_j$ onto $Y$ along $R_j$ by gluing its meridian along $\ell_j$. In this way, we obtain a closed manifold $M'$.

For part (1), we review our construction of $M'$, in order to see that $M'=S^3$.  Recall that starting with $W\cup N(V)\subseteq M=S^3$, we attached balls $D_1\times I$, \dots, $D_n\times I$ to the spheres $S_{e_1}$, \dots, $S_{e_n}$, and then attached solid tori $Z_1$, \dots, $Z_p$, to the tori $R_1$, \dots, $R_p$ by gluing meridians of $Z_1$, \dots, $Z_p$ to the curves $\ell_1$, \dots, $\ell_p$.  Since every sphere in $S^3$ bounds a ball on both sides, $Y=W\cup N(V)\cup (D_1\times I) \cup\dots\cup (D_n\times I)\subseteq S^3$.  Now each of the tori $R_1$, \dots, $R_p$ either bounds a knot complement or a solid torus in $S^3-Y$.  Also, since each $\ell_j$ is isotopic to $\lambda_j$, it bounds a surface in $S^3-Y$.  Hence replacing these knot complements or solid tori, by the solid tori $Z_1$, \dots, $Z_p$ to $\partial Y$ by gluing their meridians to the curves $\ell_1$, \dots, $\ell_p$ again yields $S^3$.  Thus $M'=S^3$.

Finally, for part (2), we suppose that $M$ is a homology sphere.  Then every sphere in $M$ bounds a homology ball (which may also be a ball) on both sides.  Whether or not a given ball $D_i\times I$ replaces a ball or a homology ball in $M-(W\cup N(V))$, it follows that $Y=W\cup N(V)\cup (D_1\times I) \cup\dots\cup (D_n\times I)$ is contained in a  homology sphere (which may actually be $S^3$).  Now by applying Remark \ref{SewTorus} we see that adding each solid torus $Z_i$ by gluing its meridian to $\ell_i$, again gives a homology sphere.\end{proof}\medskip

Now we use Part (3) of Theorem~\ref{RigidHS} to prove Proposition \ref{K7}.

\begin{proposition}\label{K7} Let $\sigma$ denote the automorphism $(123)$ of the complete graph $K_7$ with vertices numbered $1$, \dots, $7$.  Then $\sigma$ is not realizable in any orientable, closed, connected, irreducible $3$-manifold.
\end{proposition}

\begin{proof}[Proof of Proposition \ref{K7}]
Suppose that $\sigma=(123)$ of the complete graph $K_7$ is realizable in some orientable, closed, connected, irreducible $3$-manifold $M$.  Since $K_7$ is $3$-connected, we can apply the Rigidity Theorem  (\ref{RigidHS}) to get an embedding  $\Gamma$ of $K_7$ in some $3$-manifold $M'$ such that $\sigma$ is induced on $\Gamma$ by a finite order homeomorphism $h$ of $(M',\Gamma)$.

We can express the order of $h$ as $3^rq$ where $q$ is not divisible by $3$.  Since $h$ induces $\sigma=(123)$ on $\Gamma$, so does $g=h^q$.  Furthermore, $g$ has order $3^r$ and hence is orientation preserving.  Let $F$ denote the fixed point set of $g$.  Then $F$ contains the $K_4$ subgraph with vertices $4$, $5$, $6$, $7$.  But, since $g$ is orientation preserving and has finite order, by Smith Theory \cite{S39}, either $F=\emptyset$ or $F$ is a collection of disjoint $S^1$s.  By this contradiction we conclude that $\sigma$ could not have been realizable in any orientable, closed, connected, irreducible $3$-manifold.  \end{proof}

\medskip

\section{Realizable automorphisms in homology spheres}\label{CompareHS}

Using Theorem~\ref{RigidHS} together with Smith Theory \cite{S39}, many results about the symmetries of spatial graphs in $S^3$ can be generalized to homology spheres. In particular, Corollary 1 of \cite{F95}, Theorem 1 of \cite{FL02}, Theorem 2.1 of \cite{JW00}, Theorem 1 of \cite{FW96}, and various results from \cite{F89} can now be proved for homology spheres.  For example, the following result is a consequence of Theorem~\ref{RigidHS}, Smith Theory \cite{S39}, and Lemmas 1-7 and Theorem 2 of \cite{F95}.

\begin{theorem} \label{ClassThmHS}
Let $M$ be a homology sphere which has an orientation reversing homeomorphism.  Then an automorphism $\sigma$ of the complete graph $K_n$, with $n>6$, is realizable in $M$ if and only if $\sigma$ is realizable in $S^3$.
\end{theorem}

By contrast with Theorem \ref{ClassThmHS}, we prove the following.

\begin{theorem}\label{poincare} There exists an automorphism of a graph which is realizable in the Poincar\'e homology sphere but is not realizable in $S^3$.
\end{theorem}

We will use the construction of the Poincar\'e homology sphere $P$ by surgery on the link in $S^3$ in Figure \ref{PSurgery}, where each component $\alpha_{k}$ has surgery coefficient $k-1$.  In particular, let $N_2$, $N_3$, and $N_5$ be disjoint tubular neighborhoods of $\alpha_2, \alpha_3,$ and $\alpha_5$, respectively.  We then remove the interiors of the $N_k$ and sew in solid tori $N_2', N_3'$, and $N_5'$ along $\partial N_2, \partial N_3$, and $\partial N_5$ by gluing a meridian of $\partial N_k'$ to a $(k-1, 1)$ curve on $\partial N_k$.

\begin{figure}[h]
	\begin{center}
		\includegraphics[scale=.29]{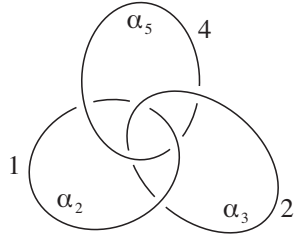}
	\end{center}
	\caption{A surgery that produces the Poincar\'e homology sphere $P$ from $S^3$.}
	\label{PSurgery}
\end{figure}

Because of the linking of $\alpha_2, \alpha_3$ and $\alpha_5$, we can choose a Hopf fibration of $S^3$ such that each $\alpha_k$ is a fiber and each $\partial N_k = \partial N_k'$ is a union of fibers. For each $k$, let $\beta_k$ be the core of $N_k'$; then, viewing $N_k' - \beta_k$ as a product $\partial N_k' \times [0,1)$, we can extend the fibration of $\partial N_k'$ to $N_k' - \beta_k$. Finally, adding in $\beta_2, \beta_3$, and $\beta_5$ as exceptional fibers, we get a Seifert fibration of $P$.  For each $k$, a Hopf fiber on $\partial N_k$ is a $(-1,1)$-curve, and its image in $\partial N_k'$ under the gluing map  is a $(1, k)$-curve. Thus, in our Seifert fibration of $N_k' \subseteq P$, each ordinary fiber goes around the exceptional fiber $\beta_k$ once and along it $k$ times, and hence $\beta_k$ has Seifert invariant $(k,1)$.

\begin{figure}[h!]
	\begin{center}
		\includegraphics[scale= .45]{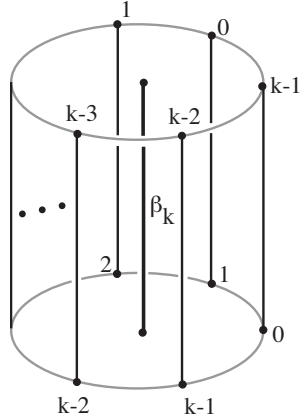}
	\end{center}
	\caption{The Seifert fibration of the solid tori $N_k'$.}
	\label{ExFiber}
\end{figure}

Note that the solid tori $N_k'$ can also be obtained by starting with a solid tube $D \times I$, where $D$ is the unit disk in $\C$ and $I$ is the unit interval, and identifying the disk $D \times \{0\}$ to $D \times \{1\}$ after rotating $D \times \{1\}$ counterclockwise by $\tfrac{2\pi}{k}$ (see Figure \ref{ExFiber}). With this identification, the image of $\{0\} \times I$ is the exceptional fiber $\beta_k$, and for each $z \in D - \{0\}$, the image of the union of $k$ vertical segments $\{z, e^{\tfrac{2\pi i}{k}} z, \cdots, e^{\tfrac{2(k-1)\pi i}{k}} z \} \times I$ is an ordinary fiber.

We now define a graph $G_{61}$ which consists of 61 vertices in cycles $\gamma_0 = \overline{a_1a_2 \cdots a_{30}}, \gamma_2 = \overline{b_1b_2 \cdots b_{15}}$, $\gamma_3 = \overline{c_1c_2 \cdots c_{10}},$ and $\gamma_5 = \overline{d_1d_2\cdots d_6}$, together with edges $\overline{a_ib_j}$ whenever $i \equiv j \pmod{15}$, $\overline{a_ic_j}$ whenever $i \equiv j \pmod {10}$, and $\overline{a_id_j}$ whenever $i \equiv j \pmod 6$. For $k = 2,3,5$, each vertex on the cycle $\gamma_k$ is connected to $k$ ``evenly-spaced" vertices on $\gamma_0$. See Figure \ref{G61} for an illustration of part of $G_{61}$, which shows the four $\gamma_k$ and the edges adjacent to vertices $a_1, b_1, c_1,$ and $d_1$.

\begin{figure}[h]
	\begin{center}
		\includegraphics[scale=.35]{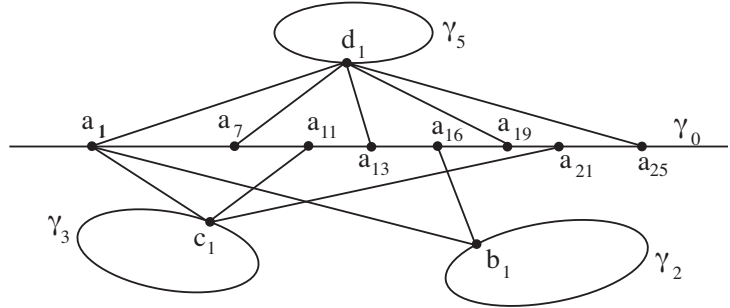}
	\end{center}
	\caption{An illustration of part of the graph $G_{61}$}
	\label{G61}
\end{figure}

Note that $G_{61}$ is non-planar and 3-connected. Now consider the automorphism
\[	\sigma = (a_1a_2 \cdots a_{30})(b_1b_2 \cdots b_{15})(c_1c_2 \cdots c_{10})(d_1d_2\cdots d_6)	\]
of $G_{61}$. Then $\sigma$ has order 30, and for $k = 2,3,5$, $\sigma$ rotates the cycle $\gamma_k$ with order $\frac{30}{k}$.

\begin{lemma} \label{SigmaP}
The automorphism $\sigma=(a_1a_2 \cdots a_{30})(b_1b_2 \cdots b_{15})(c_1c_2 \cdots c_{10})(d_1d_2\cdots d_6)$ of $G_{61}$ is realizable in the Poincar\'e homology sphere.
\end{lemma}

\begin{proof}  Let $P$ denote the Poincar\'e homology sphere with exceptional fibers $\beta_2, \beta_3$, and $\beta_5$ as described above.  Let $f$ be an order 30 homeomorphism of $P$ that rotates each ordinary fiber with order $30$ and rotates each exceptional fiber $\beta_k$ with order $\frac{30}{k}$.

Now, we use $f$ to construct an embedding of $G_{61}$ in $P$. In particular, let $\beta_0$ denote an ordinary fiber in $P$ which is disjoint from the neighborhoods $N_k'$ for $k=2,3,5$.  We embed the cycle $\gamma_0$ as $\beta_0$ such that $f$ realizes the automorphism $(a_1a_2 \cdots a_{30})$ of $\gamma_0$. Then, for each $k = 2,3,5$, we embed the cycle $\gamma_k$ as $\beta_k$, such that $f$ induces the automorphisms $(b_1b_2 \cdots b_{15})$ of $\gamma_2$, $(c_1c_2 \cdots c_{10})$ of $\gamma_3$, and $(d_1d_2\cdots d_6)$ of $\gamma_5$.

In order to embed the remaining edges, we consider the quotient space $\overline{P}$ of $P$ under $f$, with quotient map $q:P \to \overline{P}$. Let $a = q(a_1), b = q(b_1), c = q(c_1)$, and $d = q(d_1)$. We take paths $\overline{ab}, \overline{ac}$, and $\overline{ad}$ in $\overline{P}$ whose interiors are pairwise disjoint and disjoint from the loops $q(\alpha_k)$, for $k = 0,2,3,5$. Then, we embed the edges with one vertex on $\alpha_2$, $\alpha_3$, or $\alpha_5$ and the other vertex on $\alpha_0$ as the lifts of $\overline{ab}, \overline{ac}$, and $\overline{ad}$. Figure \ref{G61Embed} illustrates part of the embedding of $G_{60}$ in a neighborhood of the regular fiber $\alpha_0$, and in the neighborhood $N_5'$ of the exceptional fiber $\beta_5$ (with the same gluing of the top and bottom disk as in Figure \ref{ExFiber}).  It follows from our construction that $f$ realizes the automorphism $\sigma$ of $G_{61}$ in $P$.
\end{proof}

\begin{figure}[h]
	\begin{center}
		\includegraphics[scale=.3]{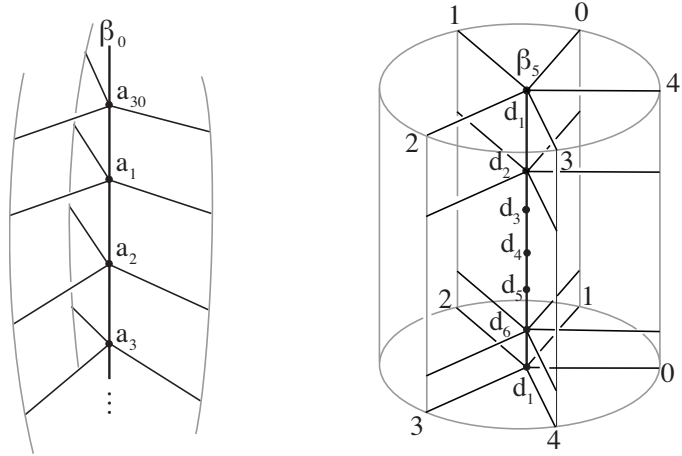}
	\end{center}
	\caption{The embedding of $G_{61}$ in neighborhoods of the fibers $\beta_0$ and $\beta_5$}
	\label{G61Embed}
\end{figure}

\medskip

To show that the automorphism $\sigma$ of $G_{61}$ is not realizable in $S^3$, we use the following Lemma.

\begin{lemma}[\cite{FLMPV14}] \label{2Circle}
Let $h$ be a finite order homeomorphism of $S^3$ which is fixed point free. Then there are at most two circles which are the fixed point set of some power of $h$ less than its order.
\end{lemma}

\begin{lemma} \label{G61NS3}
The automorphism $\sigma=(a_1a_2 \cdots a_{30})(b_1b_2 \cdots b_{15})(c_1c_2 \cdots c_{10})(d_1d_2\cdots d_6)$  of $G_{61}$ is not realizable in $S^3$.
\end{lemma}

\begin{proof}  Suppose that for some embedding $\Gamma$ of $G_{61}$ in $S^3$, there is a homeomorphism $h$ of $(S^3, \Gamma)$ that realizes $\sigma$. Since $G_{61}$ is 3-connected, by Theorem \ref{RigidS3}, we can assume that $h$ has finite order. Then $h^{30}$ pointwise fixes $\Gamma$. Since $G_{61}$ is non-planar, by Smith Theory~\cite{S39}, $G_{61}$ cannot be contained in the fixed point set of any non-trivial finite order homeomorphism of $S^3$.  It follows that $h$ must have order $30$.

For each $k$, let $\beta_k$ be the embedding of $\gamma_k$ in $\Gamma$.  Since $h$ induces $\sigma$ on $\Gamma$, $h$ rotates $\beta_k$ with order $\frac{30}{k}$ for $k=2, 3, 5$ and $h$ rotates $\beta_0$ with order $30$.   Suppose for the sake of contradiction that $h$ fixes some point $p$ in $S^3$.  Since $h$ non-trivially rotates each $\beta_k$, the fixed point $p$ cannot be on any $\beta_k$.  Now $h^{10}$ fixes $p$ and pointwise fixes $\beta_3$.   By Smith Theory this means $h^{10}$ must pointwise fix a sphere, and hence $h^{10}$ must be a reflection.  But this is impossible since $h^{10}$ has order $3$.  Thus $h$ is fixed point free.

 But by Lemma \ref{2Circle}, there are at most two circles which are the fixed point set of some power of $h$ less than $30$. Hence no such $h$ exists, and thus $\sigma$ cannot be realizable in $S^3$.
\end{proof}\medskip

Theorem \ref{poincare} now follows from Lemmas~\ref{SigmaP} and \ref{G61NS3}.
\medskip

\section{Acknowledgments} The first author thanks Daryl Cooper for a helpful conversation which led to Theorem~\ref{K7}.  The authors also thank the anonymous referee for identifying parts of the arguments that should be clarified.

\bibliographystyle{amsplain}
\bibliography{HSPaperRef.bib}

\providecommand{\bysame}{\leavevmode\hbox to3em{\hrulefill}\thinspace}
\providecommand{\MR}{\relax\ifhmode\unskip\space\fi MR }
\providecommand{\MRhref}[2]{%
  \href{http://www.ams.org/mathscinet-getitem?mr=#1}{#2}
}
\providecommand{\href}[2]{#2}
\begin{thebibliography}{10}

\bibitem{B83}
Francis Bonahon, \emph{Cobordism of automorphisms of surfaces}, Ann. Sci.
  {\'E}cole Norm. Sup. (4) \textbf{16} (1983), no.~2, 237--270. \MR{732345}

\bibitem{B02}
\bysame, \emph{Geometric structures on 3-manifolds}, Handbook of geometric
  topology, North-Holland, Amsterdam, 2002, pp.~93--164. \MR{1886669}

\bibitem{F89}
Erica Flapan, \emph{Symmetries of {M}{\"o}bius ladders}, Math. Ann.
  \textbf{283} (1989), no.~2, 271--283. \MR{980598}

\bibitem{F95}
\bysame, \emph{Rigidity of graph symmetries in the {$3$}-sphere}, J. Knot
  Theory Ramifications \textbf{4} (1995), no.~3, 373--388. \MR{1347360}

\bibitem{FLMPV14}
Erica Flapan, Nicole Lehle, Blake Mellor, Matt Pittluck, and Xan Vongsathorn,
  \emph{Symmetries of embedded complete bipartite graphs}, Fund. Math.
  \textbf{226} (2014), no.~1, 1--16. \MR{3208292}

\bibitem{FL02}
Erica Flapan and David~Linnan Li, \emph{Asymmetric two-colourings of graphs in
  {$S^3$}}, Math. Proc. Cambridge Philos. Soc. \textbf{132} (2002), no.~2,
  267--280. \MR{1874217}

\bibitem{FW96}
Erica Flapan and Nikolai Weaver, \emph{Intrinsic chirality of {$3$}-connected
  graphs}, J. Combin. Theory Ser. B \textbf{68} (1996), no.~2, 223--232.
  \MR{1417798}

\bibitem{H}
A.~Hatcher, \emph{Notes on basic 3-manifold topology},
  https://www.math.cornell.edu/~hatcher/3M/3M.pdf.

\bibitem{JS79}
William~H. Jaco and Peter~B. Shalen, \emph{Seifert fibered spaces in
  {$3$}-manifolds}, Mem. Amer. Math. Soc. \textbf{21} (1979), no.~220,
  viii+192. \MR{539411}

\bibitem{JW00}
Boju Jiang and Shicheng Wang, \emph{Achirality and planarity}, Commun. Contemp.
  Math. \textbf{2} (2000), no.~3, 299--305. \MR{1776983}

\bibitem{J79}
Klaus Johannson, \emph{Homotopy equivalences of {$3$}-manifolds with
  boundaries}, Lecture Notes in Mathematics, vol. 761, Springer, Berlin, 1979.
  \MR{551744}

\bibitem{M73}
G.~D. Mostow, \emph{Strong rigidity of locally symmetric spaces}, Princeton
  University Press, Princeton, N.J.; University of Tokyo Press, Tokyo, 1973,
  Annals of Mathematics Studies, No. 78. \MR{0385004}

\bibitem{S39}
P.~A. Smith, \emph{Transformations of finite period. {II}}, Ann. of Math. (2)
  \textbf{40} (1939), 690--711. \MR{0000177}

\bibitem{T82}
William~P. Thurston, \emph{Three-dimensional manifolds, {K}leinian groups and
  hyperbolic geometry}, Bull. Amer. Math. Soc. (N.S.) \textbf{6} (1982), no.~3,
  357--381. \MR{648524}

\bibitem{Thu}
\bysame, \emph{Three-dimensional geometry and topology. {V}ol. 1}, Princeton
  Mathematical Series, vol.~35, Princeton University Press, Princeton, NJ,
  1997, Edited by Silvio Levy. \MR{1435975}

\bibitem{W68}
Friedhelm Waldhausen, \emph{On irreducible {$3$}-manifolds which are
  sufficiently large}, Ann. of Math. (2) \textbf{87} (1968), 56--88.
  \MR{0224099}

\end{thebibliography}

\end{document}